\documentclass[11pt]{article}
\usepackage[numbers,sort&compress]{natbib}
\usepackage{enumerate}
\usepackage{tikz}
\usepackage{pgfplots}
\usepackage{amscd}
\usepackage{amsmath}
\usepackage{latexsym}
\usepackage{amsfonts}
\usepackage{amssymb}
\usepackage{amsthm}
\usepackage{verbatim}
\usepackage{mathrsfs}
\usepackage{enumerate}
\usepackage{hyperref}

\theoremstyle{plain}\newtheorem{definition}{Definition}[section]
\theoremstyle{definition}\newtheorem{theorem}{Theorem}[section]
\theoremstyle{plain}\newtheorem{lemma}[theorem]{Lemma}
\theoremstyle{plain}
\theoremstyle{plain}
\theoremstyle{remark}
\usepackage{xcolor}

\newcommand{\B}{\Big}

\newcommand{\s}{\mathrm{div}}

\newcommand{\be}{\begin{equation}}
\newcommand{\ee}{\end{equation}}
 \newcommand{\ba}{\begin{aligned}}
 \newcommand{\ea}{\end{aligned}}

  \newcommand{\f}{\frac}
    
  \newcommand{\ben}{\begin{enumerate}}
   \newcommand{\een}{\end{enumerate}}

\newcommand{\Rmnum}[1]{\expandafter\@slowromancap\romannumeral #1@}

\allowdisplaybreaks

\numberwithin{equation}{section}
\begin{document}
 \title{On energy equality   of the 3D anisotropic Navier-Stokes equations}
 \author{Yulin Ye\footnote{School of Mathematics and Statistics, Henan University, Kaifeng, Henan  475004, P. R. China. Email: ylye@vip.henu.edu.cn },\;~~Wei Wei\footnote{ School of Mathematics and Center for Nonlinear Studies,  Northwest University, Xi'an, Shaanxi 710127,  P. R. China  Email: ww5998198@126.com }
  \;~~and~~Yanqing Wang\footnote{Corresponding author. School of Mathematics and   Information Science, Zhengzhou University of Light Industry, Zhengzhou, Henan  450002,  P. R. China Email: wangyanqing20056@gmail.com}
   }
\date{}
\maketitle
\begin{abstract}
In the spirit of the recent work by Demmel and Wiedemann \cite{[DW]}, we investigate the validity of the energy balance law for weak solutions to the three-dimensional Navier-Stokes equations with only horizontal dissipation. We prove that famous  Lions-Shinbrot type energy conservation criteria, originally established for the classical isotropic Navier-Stokes equations, continue to hold in this anisotropic setting. Furthermore, we present two distinct approaches that handle  the torus case and the whole space case, respectively.

 \end{abstract}
\section{Introduction}
\label{intro}
\setcounter{section}{1}\setcounter{equation}{0}
 In geophysical fluids, the Navier-Stokes equations feature anisotropic viscosities of the form  $-\nu\Delta_{h}-\epsilon\partial^{2}_{33}$ (see \cite{[Pedlosky]}). Here, $\epsilon$ is usually much smaller than $\nu$, and the  Ekman boundary layers    arise  in rotating fluid when $\epsilon\rightarrow0$ (see \cite{[Pedlosky]}).
In this paper, we are concerned with the following 3D Navier-Stokes equations subject solely to horizontal dissipation
 \begin{equation}\left\{\begin{aligned}\label{anNS}
&u_{t}+u\cdot \nabla u-\Delta_{h}u+\nabla \Pi=0,~~\text{div}\, u=0,\\
&u|_{t=0}=u_{0}(x),
\end{aligned}\right.\end{equation}
where $u$ represents velocity field and $\Pi$ stands for the pressure, and $\Delta_h u=\partial^{2}_{x_1x_1} u+\partial^{2}_{x_2x_2} u$.
 The
initial  velocity $u_0$ satisfies   $\text{div}\,u_0=0$.

Very recently, Demmel and   Wiedemann \cite{[DW]}  proved     the energy equality for  weak solutions of the two-dimensional Navier-Stokes equations with only horizontal viscosity   at the natural energy level,  where weak solutions are lack of the regularity of   $\partial_{2}u_{1}$ in the 2D anisotropic incompressible Navier-Stokes system.  To this end, some new ingredients were provided in \cite{[DW]}, where  they  split   the energy equality into vertical component part and  horizontal component part. In addition, the divergence-free condition is fully utilized and
  an   anisotropic  commutator estimate of the
DiPerna-Lions theory for divergence-free field  was established. In the spirit of Demmel and   Wiedemann's work \cite{[DW]}, we study the
validity of energy balance law
\begin{equation}\label{ei}
\f12\int_{\Omega}|u(x,t)|^{2}dx+\int_{0}^{t}\|\nabla_{h} u\|^{2}_{L^{2}}ds=\f12\int_{\Omega}|u_{0}(x)|^{2}dx
\end{equation}
for weak solutions of  the tri-dimensional anisotropic  Navier-Stokes equations \eqref{anNS}, where $\Omega$ is either the whole space $\mathbb{R}^{3}$ or the periodic domain $\mathbb{T}^{3}$.  Before formulating our results, we recall the progress of  energy equality
\be\label{EENS}
 \f12\|u(T)\|_{L^{2}(\mathbb{R}^{3})}^{2}+  \int_{0}^{T}\|\nabla u\|_{L^{2}(\mathbb{R}^{3})}^{2}ds=
  \f12\|u_0\|_{L^{2}(\mathbb{R}^{3})}^{2}
 \ee for weak solutions of  the tri-dimensional isotropic  Navier-Stokes equations.

Energy conservation  of
weak solutions not only plays a crucial role in the study of weak-strong uniqueness, but also  is closed to the Onsager conjecture.    Onsager \cite{[Onsager]} asserted that
the
 regularity threshold in H\"older spaces of the  weak solutions  for the validity of the  kinetic energy  conservation in the ideal fluid is $1/3$. Since then, much effort has been devoted to identifying the minimal regularity assumptions that guarantee energy equality for weak solutions to viscous and inviscid flows, see, e.g. \cite{[Lions],[Shinbrot],[BY],[BY2],[Berselli2],[CCFS],[LS],[CL],[BC],[Zhang],[Galdi],[BCS],[BG],[BKR],[FGSW],[FW2018]} and references therein. Concerning the negative part of the Onsager's conjecture, we refer the reader to \cite{[GKN],[Isett],[JTW],[DS0],[DS1]}. In what follows, we restrict our attention to the energy equality and  recall some known  results in this direction below:
A Leray-Hopf weak solution $u$ to the Navier-Stokes equations satisfies the energy equality \eqref{EENS} if one of the following conditions holds
\begin{itemize}
\item Lions \cite{[Lions]}: $u\in L^{4}(0,T;L^{4}(\mathbb{R}^{3}));$

\item Shinbrot \cite{[Shinbrot]}:
\be\label{Shinb}
u\in L^{p}(0,T;L^{q}(\mathbb{R}^{3})),~\text{with}~\f{2}{p}+
 \f{2}{q}=1~\text{and}~q\geq 4;\ee
\item Taniuchi \cite{[Taniuchi]}, Beirao da Veiga-Yang \cite{[BY]}:
\be\label{tby}
u\in L^{p}(0,T;L^{q}(\mathbb{R}^{3})),~\text{with}~  \f{1}{p}+
 \f{3}{q}=1 ~\text{and}~ 3<q< 4;\ee
\item  Cheskidov-Constantin-Friedlander-Shvydkoy \cite{[CCFS]}: $u\in L^{3}(0,T;B^{\f13}_{3,\infty}(\mathbb{R}^{3}));$
\item Cheskidov-Luo \cite{[CL]}:   $u\in L^{\beta,\infty}(0,T;B^{\f2\beta+\f2p-1}_{p,\infty}(\mathbb{R}^{3})),$~with~$\f2p+\f1\beta<1$~and~$1\leq\beta<p\leq\infty;$
 \item Berselli-Chiodaroli  \cite{[BC]}, Beirao da Veiga-Yang \cite{[BY2]}, Zhang \cite{[Zhang]}:
\be\label{bcz}
\nabla u \in L^{p}\left(0, T ; L^{q}\left(\mathbb{R}^{3}\right)\right),~\text{with}~
\frac{1}{p}+\frac{3 }{q}=2~\text{and}~\frac{3 }{2}<q<\frac{9}{5},~\text{or}~
\frac{1}{p}+\frac{6}{5 q}=1~\text{and}~ q\geq\frac{9}{5}.\ee
   \end{itemize}
Inspired by  Demmel and   Wiedemann's work \cite{[DW]}, 
 a natural  question   arises whether the well-known properties of the isotropic Navier-Stokes equations continue to hold in this anisotropic setting. Compared with the isotropic counterpart, an extra difficulty stems from the lack of regularity for \(\partial_3 u_h\) in the anisotropic Navier-Stokes equations \eqref{anNS}, which further complicates the analysis of the anisotropic Navier-Stokes equations. For mathematical results on the anisotropic Navier-Stokes system, we refer the reader to \cite{[CDGG],[CW],[ZW],[CG],[YJW], [BFWZ],[CW1],[ZZ],[JTW],[F],[XZ]} and the references therein. In the present work, we concentrate on the energy equality problem and aim to address whether classical Lions-Shinbrot type criteria    for kinetic energy conservation  remain valid for the three-dimensional anisotropic Navier-Stokes equations \eqref{anNS}.

In what follows, we formulate our criteria for kinetic energy conservation of Leray-Hopf weak solutions to the anisotropic Navier-Stokes equations \eqref{anNS} with partial dissipation for the torus domain $\mathbb{T}^3$.
\begin{theorem}\label{the1.1} The energy equality of Leray-Hopf weak solutions $u$ to the 3D anisotropic Navier-Stokes equations \eqref{anNS} is valid if one of the following three conditions is satisfied
 \begin{enumerate}[(1)]
 \item  $u \in L^{4}(0,T;L^{4}(\mathbb{T}^{3}));$
 \item $u\in L^{p}(0,T;L^{q}(\mathbb{T}^{3})),~\text{with}~\f{2}{p}+
 \f{2}{q}=1~\text{and}~q\geq 4;$
 \item  $\nabla u \in L^{p}(0, T ; L^{q}(\mathbb{T}^{3})),~\text{with}~
 \frac{1}{p}+\frac{6}{5 q}=1~\text{and}~q\geq \frac{9}{5}.$
\end{enumerate}
\end{theorem}
 Now we give some comments. The   tri-dimensional  anisotropic Navier-Stokes equations \eqref{anNS}  are lack of   vertical dissipation $\partial^{2}_{x_3x_3}u$, however, the famous Lions-Shinbrot's energy conservation class originated from the standard  Navier-Stokes equations is still valid for this system. In addition, energy conservation   criteria in terms of gradient of velocity obtained in \cite{[BC],[BY2],[Zhang]} are established
for the   tri-dimensional  anisotropic Navier-Stokes equations \eqref{anNS}.
   It should be pointed out that
the proof of Theorem \ref{the1.1} is inspired by Demmel and   Wiedemann's recent work \cite{[DW]}. Compared with their deduction, the derivation of this theorem rests on the limiting case of
 Constantin-E-Titi type commutators shown in \cite{[NNT],[WY]}  for the torus case. 
 The starting point of Theorem \ref{the1.1}   is different from that in  \cite{[DW]}, hence, we get the energy identity under the full components of velocity belonging to $   L^{4}(0,T;L^{4}(\mathbb{T}^{3}))$ in one step rather than two steps in \cite{[DW]}.
Roughly speaking, Demmel and   Wiedemann proved that weak solutions satisfy the energy equality from
  $u_{2}\in L^{4}(0,T;L^{4}(\mathbb{R}^{2}))$ and $u_{2}u_{1}\in L^{2}(0,T;L^{2}(\mathbb{R}^{2}))$ in \cite{[DW]}, where the delicate pressure decomposition and direction-wise energy splitting are required.   Besides, Demmel and   Wiedemann's  anisotropic  commutator estimates established in \cite{[DW]} allow us to get the same results for the whole space  $\mathbb{R}^3$. Then, we state our second result as follows.
   \begin{theorem}\label{the1.2}
The energy equality of Leray-Hopf weak solutions $u$ to the 3D anisotropic Navier-Stokes equations \eqref{anNS} is valid if one of the following three conditions is satisfied
 \begin{enumerate}[(1)]
 \item  $u \in L^{4}(0,T;L^{4}(\mathbb{R}^{3}));$
 \item $u\in L^{p}(0,T;L^{q}(\mathbb{R}^{3})),~\text{with}~\f{2}{p}+
 \f{2}{q}=1~\text{and}~q\geq 4;$
 \item  $\nabla u \in L^{p}(0, T ; L^{q}(\mathbb{R}^3)),~\text{with}~
 \frac{1}{p}+\frac{6}{5 q}=1~\text{and}~q\geq \frac{9}{5}.$
\end{enumerate}
 \end{theorem}
 It should be pointed out that the Lions's energy conservation class $u \in L^{4}(0,T;L^{4}(\mathbb{R}^{2}))$ in Theorem \ref{the1.2} is still valid for the two dimensional anisotropic Navier-Stokes equations via a slight modification.
It is worth remarking that that Demmel and   Wiedemann obtained    energy conservation of weak solutions of  the two-dimensional anisotropic Navier-Stokes equations   without any extra integrability assumptions in \cite{[DW]}.

The rest of this  paper is organized as follows. In Section 2,
we present the notations and  some basic materials of mollifiers.
Several critical lemmas for the proof of Theorem \ref{the1.1} and Theorem \ref{the1.2} are also given. Section 3 is devoted to our new criteria for energy conservation of weak solutions to the anisotropic Navier-Stokes equations for both the periodic domain $\mathbb{T}^3$ and the whole space $\mathbb{R}^3$.

\section{Notations and  key auxiliary lemmas} \label{section2}

Throughout this paper, we will use the summation convention on repeated indices. $C$ will denote positive absolute constants which may be different from line to line unless otherwise stated in this paper.  For $p\in [1,\,\infty]$, the notation $L^{p}(0,\,T;X)$ stands for the set of measurable functions $f$ on the interval $(0,\,T)$ with values in $X$ and $\|f\|_{X}$ belonging to $L^{p}(0,\,T)$.

{\bf Mollifier kernel:} Let $\eta_{\varepsilon}:\mathbb{R}^{n}\rightarrow \mathbb{R}$ be a standard mollifier, i.e. $\eta(x)=C_0e^{-\frac{1}{1-|x|^2}}$ for $|x|<1$ and $\eta(x)=0$ for $|x|\geq 1$, where $C_0$ is a constant such that $\int_{\mathbb{R}^n}\eta (x) dx=1$. For $\varepsilon>0$, we define the rescaled mollifier $\eta_{\varepsilon}(x)=\frac{1}{\varepsilon^n}\eta(\frac{x}{\varepsilon})$ and for  any function $f\in L^1_{loc}(\mathbb{R}^n)$, its mollified version is defined as
$$f^\varepsilon(x)=(f*\eta_{\varepsilon})(x)=\int_{\mathbb{R}^n}f(x-y)\eta_{\varepsilon}(y)dy,\ \ x\in \mathbb{R}^n.$$
Next, we collect  some lemmas  which will be used in the present paper.
\begin{lemma}(\cite{[YWW]})\label{lem2.3}
	Let $ p, q, p_2, q_2\in[1,+\infty)$ and $p_1, q_1\in[1,+\infty]$  with
	$\frac{1}{p}=\frac{1}{p_1}+\frac{1}{p_2},\frac{1}{q}=\frac{1}{q_1}+\frac{1}{q_2} $. Assume $f\in L^{p_1}(0,T;L^{q_1}(\Omega)) $ and $g\in
	L^{p_2}(0,T;L^{q_2}(\Omega))$, then   there holds
	\begin{equation}\label{a4}
		\lim_{\varepsilon\rightarrow0}\| (f g)^{\varepsilon}  - f^{\varepsilon}   g ^{\varepsilon} \|_{L^p(0,T;L^q(\Omega))}= 0.
	\end{equation}
	
\end{lemma}
\begin{lemma}(\cite{[NNT]})\label{lem2.1}
Suppose that $f\in L^{p}(0,T;L^{q}(\mathbb{T}^{n}))$. Then for any $\varepsilon>0$, there holds
\be\label{2.1}
\|\nabla f^{\varepsilon}\|_{L^{p}(0,T;L^{q}(\mathbb{T}^{n}))}
\leq C\varepsilon^{-1}\| f\|_{L^{p}(0,T;L^{q}(\mathbb{T}^{n}))},
\ee
and, if $p,q<\infty$
\be\label{2.2}
\limsup_{\varepsilon\rightarrow0} \varepsilon\|\nabla f^{\varepsilon}\|_{L^{p}(0,T;L^{q}(\mathbb{T}^{n}))}=0.
\ee

\end{lemma}

\begin{lemma} (\cite{[NNT],[WY]})\label{lem2.7}   Let $1\leq p,q,p_1,p_2,q_1,q_2\leq \infty$  with $\frac{1}{p}=\frac{1}{p_1}+\frac{1}{p_2}$ and $\frac{1}{q}=\frac{1}{q_1}+\frac{1}{q_2}$. Assume that $f\in L^{p_1}(0,T;W^{1,q_1}(\mathbb{T}^n))$ and $g\in L^{p_2}(0,T;L^{q_2}(\mathbb{T}^n))$. Then for any $\varepsilon> 0$, there holds
		\begin{align} \label{fg'}
		\|(fg)^\varepsilon-f^\varepsilon g^\varepsilon\|_{L^p(0,T;L^q(\mathbb{T}^n))}\leq C\varepsilon \|f\|_{L^{p_1}(0,T;W^{1,q_1}( \mathbb{T}^n))}\|g\|_{L^{p_2}(0,T;L^{q_2}(\mathbb{T}^n))}.
		\end{align}
		Moreover, if $p_2,q_2<\infty$, then
		\begin{align}\label{limite'}
		\limsup_{\varepsilon \to 0}\varepsilon^{-1} \|(fg)^\varepsilon-f^\varepsilon g^\varepsilon\|_{L^p(0,T;L^q(\mathbb{T}^n))}=0.
		\end{align}
	\end{lemma}

The next lemma was established in \cite{[DW]} for the two-dimensional setting $\mathbb{R}^2$ with the special exponents $p_i=q_i=2$ for $i=1,2$. Now, we generalize the result by adopting space-time integral norms with distinct exponents and extending the spatial dimension from two to three.

\begin{lemma}
	\label{pLions}  Assume that the exponents satisfy $$1\leq p_1,q_1\leq \infty,~1\leq p, q, p_k, q_k< \infty,~k=2,3,4,$$ and the scaling relations $$\frac{1}{p}=\frac{1}{p_1}+\frac{1}{p_2}=\frac{1}{p_3}+\frac{1}{p_4}~ and ~ \frac{1}{q}=\frac{1}{q_1}+\frac{1}{q_2}=\frac{1}{q_3}+\frac{1}{q_4}.$$ Suppose that $$\partial_{i}f\in L^{p_2}(0,T;L^{q_2}(\mathbb{R}^3)), ~\text{and}~f\in L^{p_4}(0,T;L^{q_4}(\mathbb{R}^3)),$$
	and the divergence-free vector field $u$ fulfills
	 $$  u\in L^{p_1}(0,T;L^{q_1}(\mathbb{R}^3)),~\text{and}~\partial_{i} u\in L^{p_3}(0,T;L^{q_3}(\mathbb{R}^3)),$$
	 for all $i=1, 2,\cdots,n-1.$
Then,  $$  \s[(uf)^\varepsilon  -   u  f^\varepsilon   ]\to 0\quad\text{ strongly~in } {L^{p}(0,T;L^{q}(\mathbb{R}^3))},~\text{as}~\varepsilon\to 0.$$
\end{lemma}
\begin{proof}
First, by direct computation, we deduce that
	\be\label{2.5}\ba \s[(uf)^\varepsilon  -   u  f^\varepsilon   ]  =&  \sum_{i=1}^{  n-1}\partial_{i}[(u_{i}f)^\varepsilon  -   u_{i}  f^\varepsilon   ] +\partial_{n}[(u_{n}f)^\varepsilon  -   u_{n}  f^\varepsilon   ]\\ = & \sum_{i=1}^{  n-1}[\partial_{i}(u_{i}f)^\varepsilon  -   u_{i}  \partial_{i}f^\varepsilon   ] +[\partial_{n}(u_{n}f)^\varepsilon  -   u_{n}\partial_{n}  f^\varepsilon   ]\\ = & \sum_{i=1}^{  n-1}[(\partial_{i}u_{i}f)^\varepsilon + (u_{i}\partial_{i}f)^\varepsilon-   u_{i}  \partial_{i}f^\varepsilon   ] +[\partial_{n}(u_{n}f)^\varepsilon  -   u_{n}\partial_{n}  f^\varepsilon   ]\\
	=&R_1^\varepsilon +R_2^\varepsilon,\ea\ee
	where the divergence-free condition $\s u=0$ has been used.
	
	For the first term $R_1^\varepsilon$, it can be reformulated as
	\be\label{2.6}
	\ba
	R_1^\varepsilon&=\sum_{i=1}^{  n-1}[(\partial_{i}u_{i}f)^\varepsilon + (u_{i}\partial_{i}f)^\varepsilon-   u_{i}  \partial_{i}f^\varepsilon   ]\\
	&=\sum_{i=1}^{  n-1}(\partial_i u_i f)^\varepsilon+ \sum_{i=1}^{  n-1}[(u_i \partial_i f)^\varepsilon -(u_i \partial_i f)+(u_i \partial_i f)-u_i \partial_i f^\varepsilon].
		\ea\ee
Regarding the second term on the right-hand side of \eqref{2.6}, it follows from the H\"older inequality that
\be\label{2.7}\ba
\|u_i \partial_i f\|_{L^p(0,T;L^q(\mathbb{R}^3))}\leq \|u_i\|_{L^{p_1}(0,T;L^{q_1}(\mathbb{R}^3))}\|\partial_i f\|_{L^{p_2}(0,T;L^{q_2}(\mathbb{R}^3))},
\ea\ee
which together with the standard properties of the mollification yields
\be\label{2.8}\ba
\|(u_i \partial_i f)^\varepsilon - u_i \partial_i f\|_{L^p(0,T;L^q(\mathbb{R}^3))}\rightarrow 0,~~\text{as}~~\varepsilon \to 0,
 \ea\ee	
and
\be\label{2.9}\ba
&\|(u_i \partial_i f)-u_i \partial_i f^\varepsilon\|_{L^p(0,T;L^q(\mathbb{R}^3))}\\
\leq& \|u_i\|_{L^{p_1}(0,T;L^{q_1}(\mathbb{R}^3))}\|\partial_i f^\varepsilon -\partial_i f\|_{L^{p_2}(0,T;L^{q_2}(\mathbb{R}^3))}\to 0, ~~\text{as}~~\varepsilon \to 0,
\ea\ee
provided that $1\leq p,q, p_1, q_1\leq \infty, 1\leq p_2,q_2<\infty$ satisfying $\f1p=\frac{1}{p_1}+\frac{1}{p_2}$ and $\frac{1}{q}=\frac{1}{q_1}+\frac{1}{q_2}$.

Consequently, in combination with \eqref{2.6}, \eqref{2.8} and \eqref{2.9}, we conclude that
\be\label{2.10}
R_1^\varepsilon \rightarrow \sum_{i=1}^{n-1} \partial_i u_i f~~\text{strongly~in}~L^p(0,T;L^q(\mathbb{R}^3)),~\text{as}~\varepsilon \to 0.
\ee
For the second term $R_2^\varepsilon$, in light of the mean value theorem, we get
\be\label{2.11}\ba
R_2^\varepsilon &=\partial_{n}(u_{n}f)^\varepsilon  -   u_{n}\partial_{n}  f^\varepsilon\\
&=\int_{\mathbb{R}^3}\partial_{y_n} \eta_\varepsilon (y)\left(u_n(x-y)-u_n(x)\right)f(x-y) dy\\
&=\f{1}{\varepsilon}\int_{\mathbb{R}^3}\partial_{z_n}\eta(z)\left(u_n(x-\varepsilon z)-u_n(x)\right) f(x-\varepsilon z) dz\\
&=-\f{1}{\varepsilon}\int_{\mathbb{R}^3}\partial_{z_n}\eta(z)\int_0^1 \varepsilon z \cdot \nabla u_{n}(x-\theta \varepsilon z) d\theta  f(x-\varepsilon z) dz\\
&=-\sum_{i=1}^n \int_{\mathbb{R}^3}\partial_{z_n}\eta(z) z_i\int_0^1 \partial_i u_{n}(x-\theta \varepsilon z) d\theta  f(x-\varepsilon z) dz.
\ea\ee
On the one hand, by virtue of the Minkowski inequality and the H\"older inequality, we deduce that
\be\label{2.12}\ba
&\|\partial_{n}(u_{n}f)^\varepsilon  -   u_{n}\partial_{n}  f^\varepsilon\|_{L^p(0,T;L^q(\mathbb{R}^3)}\\
\leq & \sum_{i=1}^n \int_{\mathbb{R}^3}\int_0^1 |\partial_{z_n}\eta (z) z_i |\|\partial_i u_n(x-\theta \varepsilon z)\|_{L^{p_3}(0,T;L^{q_3}(\mathbb{R}^3))}\|f(x-\varepsilon z)\|_{L^{p_4}(0,T;L^{q_4}(\mathbb{R}^3))} d\theta dz\\
\leq& C\|\partial_i u_n\|_{L^{p_3}(0,T;L^{q_3}(\mathbb{R}^3))}\|f\|_{L^{p_4}(0,T;L^{q_4}(\mathbb{R}^3))},
\ea\ee
where $\f{1}{p}=\f{1}{p_3}+\f{1}{p_4}$ and $\f{1}{q}=\f{1}{q_3}+\f{1}{q_4}$.\\
On the other hand, the  continuity of the translations in Lebergue space gives
\be\label{2.13}\ba
&\partial_i u_n(x-\theta \varepsilon z)\rightarrow \partial_i u_n(x)~&\text{strongly~in~} L^{p_3}(0,T;L^{q_3}(\mathbb{R}^3)), ~\text{as} ~\varepsilon \to 0,\\
&f(x-\theta \varepsilon z)\rightarrow f(x)~&\text{strongly~in~} L^{p_4}(0,T;L^{q_4}(\mathbb{R}^3)), ~\text{as} ~\varepsilon \to 0,
\ea\ee
provided that $1\leq p_3,q_3,p_4,q_4<\infty$.

Then combining \eqref{2.11},\eqref{2.12} and \eqref{2.13}, we conclude from the dominated convergence theorem that
\be\label{2.14}\ba
R_2^\varepsilon &=\partial_{n}(u_{n}f)^\varepsilon  -   u_{n}\partial_{n}  f^\varepsilon\\
&\to -\sum_{i=1}^n \left(\int_{\mathbb{R}^3} \partial_{z_n} \eta(z) z_i  dz\right) \partial_i u_n(x) f(x)~&\text{strongly~in~} L^{p}(0,T;L^{q}(\mathbb{R}^3)), ~\text{as} ~\varepsilon \to 0\\
&=\partial_n u_n(x) f(x).
\ea\ee
Here in the last equality, we have used the integration by parts and the facts that $\sum_{i=1}^n \partial_{z_n} z_i=1$ and $\int_{\mathbb{R}^3} \eta(z) dz=1$.

Thus, in combination with \eqref{2.10} and \eqref{2.14}, and using the divergence-free condition, we obtain
\be\label{2.15}\ba
R_1^\varepsilon+R_2^\varepsilon &\to (\sum_{i=1}^{n-1} \partial_i u_i+\partial_n u_n) f=0, ~&\text{strongly~in~} L^{p}(0,T;L^{q}(\mathbb{R}^3)), ~\text{as} ~\varepsilon \to 0.
\ea\ee
Thus, the proof of this lemma is completed.
\end{proof}
Finally, to end this section, we give the definition of  weak solutions to the 3D incompressible anisotropic Navier-Stokes equations.
\begin{definition}\label{weak solution}
	Let $\Omega =\mathbb{T}^3~\text{or}~\mathbb{R}^3$. We call a vector field $u$ a weak solution to the 3D incompressible ansotropic Navier-Stokes equations \eqref{anNS} if
	\begin{equation}\label{2.17}
		u\in L^\infty(0,T;L^2(\Omega)),\qquad \nabla_h u\in L^2(0,T;L^2(\Omega)),
	\end{equation}
	it is weakly divergence-free, i.e.,
$$
	\int_{\Omega} u\cdot\nabla\psi\,dx = 0
$$
	holds for all $\psi\in C_0^\infty(\Omega)$ and almost every $t$, and
	\begin{equation}\label{2.18}
		\int_0^T\int_{\Omega}
		\partial_t\phi\cdot u
		+\nabla\phi:(u\otimes u)
		-\nabla_{h}\phi\cdot\nabla_{h}u
		\,dx\,dt = 0
	\end{equation}
	holds for every divergence-free $\phi\in C_0^\infty((0,T)\times\Omega)$.
\end{definition}

 \section{Energy conservation   of  3D anisotropic Navier-Stokes equations}
In this section, we are devoted to the proof of the energy conservation of weak solutions to the 3D anisotropic Navier-Stokes equations \eqref{anNS}, for both the periodic domain  $\mathbb{T}^3$ and the whole space $\mathbb{R}^3$.
\subsection{Torus case: $\mathbb{T}^3$}\label{subsec1}
\begin{proof}[Proof of Theorem \ref{the1.1}]

 Mollifying the equations \eqref{anNS}  in spatial direction, we arrive at
\be\ba\label{rmhd}
&\partial_{t}{u^{\varepsilon}}-\Delta_{h} u^\varepsilon +   (u\cdot\nabla u)^{\varepsilon} +\nabla\Pi^{\varepsilon}= 0.
\ea\ee
Multiplying  \eqref{rmhd} by $u^{\varepsilon}$ and integrating it with respect to $x$ and $t$, we conclude by the divergence-free condition and integration by parts that

\begin{equation}\ba\label{3.2}
&\f12\|u ^{\varepsilon}(T)\|^{2}_{L^{2}(\mathbb{T}^3)}
-\f12\|{u }^{\varepsilon}(0) \|^{2}_{L^{2}(\mathbb{T}^3)} +\int_{0}^{T}\|\nabla_{h}{u }^{\varepsilon}(s) \|^{2}_{L^{2}(\mathbb{T}^3)} dt \\=& - \int_{0}^{T}\int_{\mathbb{T}^{3}} (u\cdot \nabla u_{h})^{\varepsilon}\cdot u_{h}^{\varepsilon} dxdt - \int_{0}^{T}\int_{\mathbb{T}^{3}} (u\cdot \nabla u_{3})^{\varepsilon}u_{3}^{\varepsilon} dxdt\\
=&I+II.
\ea\end{equation}
According to the divergence-free condition again, we get
$$  \int_{0}^{T}\int_{\mathbb{T}^{3}} (u^{\varepsilon}\cdot \nabla u_{h}^{\varepsilon})\cdot u_{h}^{\varepsilon} dxdt =0$$
As a result, we rewrite $I$ as
\be\label{3.3}\ba
 &- \int_{0}^{T}\int_{\mathbb{T}^{3}} (u\cdot \nabla u_{h})^{\varepsilon}\cdot u_{h}^{\varepsilon} dxdt\\
= & \int_{0}^{T}\int_{\mathbb{T}^{3}} [(u_{h}^{\varepsilon}\cdot \nabla_{h} u_{h}^{\varepsilon})-(u_{h}\cdot \nabla_{h} u_{h})^{\varepsilon}]\cdot u_{h}^{\varepsilon} dxdt+\int_{0}^{T}\int_{\mathbb{T}^{3}} [(u_{3}^{\varepsilon}\partial_{3} u_{h}^{\varepsilon})-(u_{3}\partial_{3} u_{h})^{\varepsilon}]\cdot u_{h}^{\varepsilon} dxdt\\
=&I_1+I_2.
\ea\ee
For the last term on the right-hand side of \eqref{3.3}, invoking the integration by parts and the divergence-free condition, we see that
\begin{align}
&\int_{0}^{T}\int_{\mathbb{T}^{3}} [(u_{3}^{\varepsilon}\partial_{3} u_{h}^{\varepsilon})-(u_{3}\partial_{3} u_{h})^{\varepsilon}]\cdot u_{h}^{\varepsilon} dxdt\nonumber\\
=&-\int_{0}^{T}\int_{\mathbb{T}^{3}} [(\partial_{3}u_{3}^{\varepsilon} u_{h}^{\varepsilon})-(\partial_{3} u_{3} u_{h})^{\varepsilon}]\cdot u_{h}^{\varepsilon} dxdt
+\int_{0}^{T}\int_{\mathbb{T}^{3}} \partial_{3}[(u_{3}^{\varepsilon} u_{h}^{\varepsilon})-(  u_{3} u_{h})^{\varepsilon}]\cdot u_{h}^{\varepsilon} dxdt\nonumber\\
=&-\int_{0}^{T}\int_{\mathbb{T}^{3}} [(\partial_{3}u_{3}^{\varepsilon} u_{h}^{\varepsilon})-(\partial_{3} u_{3} u_{h})^{\varepsilon}]\cdot u_{h}^{\varepsilon} dxdt
-\int_{0}^{T}\int_{\mathbb{T}^{3}} [(u_{3}^{\varepsilon} u_{h}^{\varepsilon})-(  u_{3} u_{h})^{\varepsilon}]\cdot \partial_{3}u_{h}^{\varepsilon} dxdt\nonumber\\
=& \int_{0}^{T}\int_{\mathbb{T}^{3}} [(\partial_{1}u_{1}^{\varepsilon} u_{h}^{\varepsilon}+ \partial_{2}u_{2}^{\varepsilon}u_{h}^{\varepsilon})
-(\partial_{1} u_{1} u_{h}+\partial_{2} u_{2} u_{h})^{\varepsilon}]\cdot u_{h}^{\varepsilon} dxdt\nonumber
\\&-\int_{0}^{T}\int_{\mathbb{T}^{3}} [(u_{3}^{\varepsilon} u_{h}^{\varepsilon})-(  u_{3} u_{h})^{\varepsilon}]\cdot \partial_{3}u_{h}^{\varepsilon} dxdt\nonumber\\
=& \int_{0}^{T}\int_{\mathbb{T}^{3}} [\partial_{1}u_{1}^{\varepsilon} u_{h}^{\varepsilon} -(\partial_{1} u_{1} u_{h} )^{\varepsilon}]\cdot u_{h}^{\varepsilon} dxdt+ \int_{0}^{T}\int_{\mathbb{T}^{3}} [\partial_{2}u_{2}^{\varepsilon}u_{h}^{\varepsilon}
-( \partial_{2} u_{2} u_{h})^{\varepsilon}]\cdot u_{h}^{\varepsilon} dxdt\nonumber
\\&-\int_{0}^{T}\int_{\mathbb{T}^{3}} [(u_{3}^{\varepsilon} u_{h}^{\varepsilon})-(  u_{3} u_{h})^{\varepsilon}]\cdot \partial_{3}u_{h}^{\varepsilon} dxdt\nonumber\\
=&I_{21}+I_{22}+I_{23}\label{3.4}.
\end{align}
 We assert that the following condition
\be\label{3.5}
\nabla_{h}u\in L^{p}(0,T;L^{ q}( \mathbb{T}^{3})),u\in L^{\f{2p}{p-1}}(0,T;L^{ \f{2q}{q-1}}( \mathbb{T}^{3}))\ee
implies the energy equality \eqref{ei}.\\
Actually, the incompressible condition allows us to deduce from \eqref{3.5}  that
\be\label{3.6}
\nabla_{h}u, \nabla u_{3} \in L^{p}(0,T;L^{ q}( \mathbb{T}^{3})), u\in   L^{\f{2p}{p-1}}(0,T;L^{ \f{2q}{q-1}}( \mathbb{T}^{3})).
\ee
Making use of the H\"older inequality, we discover that
\be\ba
 I_{21}=&\B|\int_{0}^{T}\int_{\mathbb{T}^{3}} [\partial_{1}u_{1}^{\varepsilon} u_{h}^{\varepsilon} -(\partial_{1} u_{1} u_{h} )^{\varepsilon}]\cdot u_{h}^{\varepsilon} dxdt\B|\\
 \leq & \|\partial_{1}u_{1}^{\varepsilon} u_{h}^{\varepsilon} -(\partial_{1} u_{1} u_{h} )^{\varepsilon}\|_{L^{\f{2p}{p+1}}(0,T;L^{ \f{2q}{q+1}}( \mathbb{T}^{3}))}  \|  u_{h}^{\varepsilon} \|_{L^{\f{2p}{p-1}}(0,T;L^{ \f{2q}{q-1}}( \mathbb{T}^{3}))}.
\ea\ee
Owing to \eqref{3.6},
we conclude by Lemma \eqref{lem2.3} that
\be
\lim_{\varepsilon\rightarrow0}\|\partial_{1}u_{1}^{\varepsilon} u_{h}^{\varepsilon} -(\partial_{1} u_{1} u_{h} )^{\varepsilon}\|_{L^{\f{2p}{p+1}}(0,T;L^{ \f{2q}{q+1}}( \mathbb{T}^{3}))}=0,
\ee
which turns out that
\be\label{3.9}
\lim_{\varepsilon\rightarrow0}I_{21}=0.
\ee
Likewise,
\be\label{3.10}
\lim_{\varepsilon\rightarrow0}I_{22}=0,
\ee
and
\be\label{3.11}
\lim_{\varepsilon\rightarrow0}I_{1}=0.
\ee

Before going further, together with $ u_{3}\in L^\infty(0,T;L^2(\mathbb{T}^3))$, we conclude that $\nabla u_{3}\in L^p(0,T;L^q(\mathbb{T}^3))$ implies $  u_{3}\in L^p(0,T;W^{1,q}(\mathbb{T}^3))$. Indeed,  for $1\leq q\leq 2$, we derive from the H\"older inequality that
 	$$\| u_{3}\|_{L^p(0,T;L^q(\mathbb{T}^3))}\leq C\| u_{3}\|_{L^p(0,T;L^2(\mathbb{T}^3))}\leq C\| u_{3}\|_{L^\infty(0,T;L^2(\mathbb{T}^3))}.$$
 	For $q>2$, we conclude by the
 	Gagliardo-Nirenberg inequality that
 	 $$\ba
 \| u_{3}\|_{L^p(0,T;L^q(\mathbb{T}^3))}\leq& C\| u_{3}\|_{L^p(0,T;L^2(\mathbb{T}^3))}^{\f{2q}{5q-6}}\|\nabla u_{3}\|_{L^p(0,T;L^q(\mathbb{T}^3))}^{\f{3q-6}{5q-6}}+C\|u\|_{L^\infty(0,T;L^2(\mathbb{T}^3))}\\
 \leq &C\| u_{3}\|_{L^\infty(0,T;L^2(\mathbb{T}^3))}^{\f{2q}{5q-6}}\|\nabla u_{3}\|_{L^p(0,T;L^q(\mathbb{T}^3))}^{\f{3q-6}{5q-6}}+C\| u_{3}\|_{L^\infty(0,T;L^2(\mathbb{T}^3))}.
 \ea$$
 Consequently,
 $\nabla u_{3} \in L^{p}(0,T;L^{ q}( \mathbb{T}^{3}))$ means that
  $  u_{3} \in L^{p}(0,T;W^{1,q}( \mathbb{T}^{3})).$ With this in hand, we continue to deal with $I_{23}.$
In view of the H\"older inequality, we observe that
\be\ba\label{3.12}
&\B| \int_0^T\int_{\mathbb{T}^{3}}    \B[(u_{3} u_{h})^{\varepsilon}- (u_3^{\varepsilon}  u_{h}^{\varepsilon})\B] \partial_{3}u_{h}^{\varepsilon}dxdt\B|\\
\leq& C\|(u_{3} u_{h})^{\varepsilon}- (u_3^{\varepsilon}  u_{h}^{\varepsilon}) \|_{L^{\f{2p}{p+1}}(0,T;L^{ \f{2q}{q+1}}( \mathbb{T}^{3}))} \|\nabla u_{h}^{\varepsilon} \|_{L^{\f{2p}{p-1}}(0,T;L^{ \f{2q}{q-1}}( \mathbb{T}^{3}))}.
\ea\ee
On the one hand, taking advantage  of
  $u\in   L^{\f{2p}{p-1}}(0,T;L^{ \f{2q}{q-1}}( \mathbb{T}^{3}))$ and Lemma \ref{lem2.1}, we notice that
\be\label{3.13}\ba
& \|\nabla u_{h}^{\varepsilon} \|_{  L^{\f{2p}{p-1}}(0,T;L^{ \f{2q}{q-1}}( \mathbb{T}^{3}))}\leq C\varepsilon^{-1}\|u_{h}\|_{  L^{\f{2p}{p-1}}(0,T;L^{ \f{2q}{q-1}}( \mathbb{T}^{3}))},\\ &\limsup_{\varepsilon\rightarrow0}\varepsilon \|\nabla u_h^{\varepsilon}\|_{  L^{\f{2p}{p-1}}(0,T;L^{ \f{2q}{q-1}}( \mathbb{T}^{3}))}=0.
\ea\ee
On the other hand, it follows from the Lemma \ref{lem2.7} and $u_{3}\in L^{p}(0,T;W^{1, q}( \mathbb{T}^{3}))$ that
\begin{equation}\label{3.14} \begin{aligned}
&\|(u_{3} u_{h})^{\varepsilon}- (u_3^{\varepsilon}  u_{h}^{\varepsilon}) \|_{L^{\f{2p}{p+1}}(0,T;L^{ \f{2q}{q+1}}( \mathbb{T}^{3}))}  \leq C\varepsilon\|u_{3}\|_{L^{p}(0,T;W^{1,  q }( \mathbb{T}^{3}))}  \| u_{h}\|_{L^{\f{2p}{p-1}}(0,T;L^{ \f{2q}{q-1}}( \mathbb{T}^{3}))}.
\end{aligned}\end{equation}
Plugging this into \eqref{3.12}, we observe that
$$\ba\B| &\int_0^T\int_{\mathbb{T}^{3}}    \B[(u_{3} u_{h})^{\varepsilon}- (u_3^{\varepsilon}  u_{h}^{\varepsilon})\B] \partial_{3}u_{h}^{\varepsilon}dxdt\B|\\
\leq& C\|(u_{3} u_{h})^{\varepsilon}- (u_3^{\varepsilon}  u_{h}^{\varepsilon}) \|_{L^{\f{2p}{p+1}}(0,T;L^{ \f{2q}{q+1}}( \mathbb{T}^{3}))}  \|\nabla u_{h}^{\varepsilon} \|_{L^{p}(0,T;L^{q}( \mathbb{T}^{3}))}\\
\leq&\varepsilon \|\nabla u_h^{\varepsilon}\|_{L^{p}(0,T;L^{q}( \mathbb{T}^{3}))},
\ea$$
which together with \eqref{3.13} leads to
\be\label{3.15}
\lim_{\varepsilon\rightarrow0}I_{23}=0.\ee
Then in combination with \eqref{3.4}, \eqref{3.9}--\eqref{3.11} and \eqref{3.15}, we have
\be\label{3.16}
\lim_{\varepsilon\rightarrow0}I=0.
\ee
Next, we are ready to deal with the remainder term $II$. Firstly, thanks to
$$\int_{0}^{T}\int_{\mathbb{T}^{3}} (u^{\varepsilon}\cdot \nabla u_{3}^{\varepsilon})u_{3}^{\varepsilon} dxdt=0,$$
we reformulate $II$ as
\be\ba
&- \int_{0}^{T}\int_{\mathbb{T}^{3}} (u\cdot \nabla u_{3})^{\varepsilon}u_{3}^{\varepsilon} dxdt\\
= & \int_{0}^{T}\int_{\mathbb{T}^{3}} [(u_{h}^{\varepsilon}\cdot \nabla_{h} u_{3}^{\varepsilon})-(u_{h}\cdot \nabla_{h} u_{3})^{\varepsilon}]u_{3}^{\varepsilon} dxdt+\int_{0}^{T}\int_{\mathbb{T}^{3}} [(u_{3}^{\varepsilon}\partial_{3} u_{3}^{\varepsilon})-(u_{3}\partial_{3} u_{3})^{\varepsilon}]u_{3}^{\varepsilon} dxdt\\
=&\int_{0}^{T}\int_{\mathbb{T}^{3}} [(u_{h}^{\varepsilon}\cdot \nabla_{h} u_{3}^{\varepsilon})-(u_{h}\cdot \nabla_{h} u_{3})^{\varepsilon}]u_{3}^{\varepsilon} dxdt\\&-\int_{0}^{T}\int_{\mathbb{T}^{3}} [(u_{3}^{\varepsilon}(\partial_{1} u_{1}^{\varepsilon}+\partial_{2} u_{2}^{\varepsilon}))-(u_{3} \partial_{1} u_{1})^{\varepsilon}-(u_{3}\partial_{2} u_{2})^{\varepsilon}]u_{3}^{\varepsilon} dxdt\\
=&\int_{0}^{T}\int_{\mathbb{T}^{3}} [(u_{h}^{\varepsilon}\cdot \nabla_{h} u_{3}^{\varepsilon})-(u_{h}\cdot \nabla_{h} u_{3})^{\varepsilon}]u_{3}^{\varepsilon} dxdt\\&-\int_{0}^{T}\int_{\mathbb{T}^{3}} [(u_{3}^{\varepsilon} \partial_{1} u_{1}^{\varepsilon}-(u_{3} \partial_{1} u_{1})^{\varepsilon}]u_{3}^{\varepsilon} dxdt-\int_{0}^{T}\int_{\mathbb{T}^{3}}[u_{3}^{\varepsilon}\partial_{2} u_{2}^{\varepsilon} -(u_{3}\partial_{2} u_{2})^{\varepsilon}]u_{3}^{\varepsilon} dxdt\\
=&II_1+II_2+II_3.
\ea\ee
By virtue of the H\"older inequality,
\be\ba|II_1|\leq \|(u_{h}^{\varepsilon}\cdot \nabla_{h} u_{3}^{\varepsilon})-(u_{h}\cdot \nabla_{h} u_{3})^{\varepsilon}\|_{L^{\f{2p}{p+1}}(0,T;L^{ \f{2q}{q+1}}( \mathbb{T}^{3}))}  \|  u_{3}^{\varepsilon} \|_{  L^{\f{2p}{p-1}}(0,T;L^{ \f{2q}{q-1}}( \mathbb{T}^{3}))}.
\ea\ee
In view of \eqref{3.6},
we deduce from Lemma \eqref{lem2.3} that
\be
\lim_{\varepsilon\rightarrow0}\|(u_{h}^{\varepsilon}\cdot \nabla_{h} u_{3}^{\varepsilon})-(u_{h}\cdot \nabla_{h} u_{3})^{\varepsilon}\|_{L^{\f{2p}{p+1}}(0,T;L^{ \f{2q}{q+1}}( \mathbb{T}^{3}))}
=0,\ee
which implies that
\be
\lim_{\varepsilon\rightarrow0}II_1=0.\ee
By the same token, we also have
\be
\lim_{\varepsilon\rightarrow0}II_2=\lim_{\varepsilon\rightarrow0}II_3=0.\ee
As a consequence, we conclude that
\be\label{3.22}
\lim_{\varepsilon\rightarrow0}II =0.\ee
Hence, combining \eqref{3.16} and \eqref{3.22}, we verified the assertment \eqref{3.5}.
Now, we are  in a position to  pass to the limit of $\varepsilon$ in \eqref{3.2} to get the energy conservation \eqref{ei}.

{\bf (1) } As the natural energy gives $\nabla_h u \in L^{2}(0,T;L^2(\mathbb{T}^3))$, by choosing $p=q=2$ in \eqref{3.5}, we 	immediately prove that the condition $u\in L^{4}(0,T;L^{4}(\mathbb{T}^3))$ yields the energy equality. It is worth remarking that
the rest proof  in Theorem \ref{the1.1} can be reduced to this special case.

{\bf (2)} Next, we deal with the case (2) in Theorem
\ref{the1.1} with $q\geq 4$ and $\frac{2}{p}+\frac{2}{q}=1$.
The  Gagliardo-Nirenberg   inequality  guarantees that
\begin{equation}
\begin{aligned}
\|u\|_{L^{4}(0,T;L^{4}(\mathbb{T}^3) )}  \leq& C\|u\|_{L^{\infty}\left(0,T;L^{2}(\mathbb{T}^{3})\right)}^{\frac{q-4}{2 q-4}}\|u\|_{L^{p}(0,T;L^q(\mathbb{T}^3))}^{\frac{q}{2q-4}}\leq C.
\end{aligned}
\end{equation}
From the result just proved, we obtain the  energy equality via (1) in Theorem \ref{the1.1}  with $q\geq 4$. \\

{\bf (3)} Now, we focus on the proof of (3). Indeed, note that $  u\in L^{p}(0,T;W^{1,q}(\mathbb{T}^{3}))$, therefore, according to \eqref{3.5},
it suffices to derive  $u\in L^{\f{2p}{p-1}}(0,T;L^{\f{2q}{q-1}}(\mathbb{T}^3)) $   from  \eqref{3.5}.  For $q\geq \frac{9}{5}$, by the Gagliardo-Nirenberg   inequality, we get
 \begin{equation}\label{c19}
 \begin{aligned}
  \|u\|_{L^{\f{2q}{q-1}}(\mathbb{T}^3)}\leq C\|u\|_{L^{2}(\mathbb{T}^3)}^{\f{5q-9}{5q-6}}\|\nabla u\|_{L^{q}(\mathbb{T}^3)}^{\f{3}{5q-6}}.
  \end{aligned}\end{equation}
  Thanks to $\frac{1}{p}+\frac{6}{5 q}=1$, we further infer that
\begin{equation}\label{c20}
\begin{aligned}
  \|u\|_{L^{\f{2p}{p-1}}(0,T;L^{\f{2q}{q-1}}(\mathbb{T}^3))}\leq C \|u\|_{L^{\infty}(0,T;L^{2}(\mathbb{T}^3))}^{\f{5q-9}{5q-6}}\|\nabla u\|_{L^{p}(0,T;L^{q}(\mathbb{T}^3))}^{\frac{3}{5q-6}}\leq C.
  \end{aligned}\end{equation}
In light of \eqref{3.5}, we have proved case (3) in Theorem \ref{the1.1} for $q\geq \frac{9}{5}$.

Thus, we complete the proof of Theorem \ref{the1.1}.
\end{proof}
\subsection{Whole space case: $\mathbb{R}^3$}\label{subsec2}
\begin{proof}[Proof of Theorem \ref{the1.2}]
Following the proof strategy of Theorem \ref{the1.1} and recalling \eqref{3.2}, we obtain
\begin{equation}\ba\label{3.26}
	&\f12\|u ^{\varepsilon}(T)\|^{2}_{L^{2}(\mathbb{R}^3)}
	-\f12\|{u }^{\varepsilon}(0) \|^{2}_{L^{2}(\mathbb{R}^3)} +\int_{0}^{T}\|\nabla_{h}{u }^{\varepsilon}(s) \|^{2}_{L^{2}(\mathbb{R}^3)} dt \\=& - \int_{0}^{T}\int_{\mathbb{R}^{3}} (u\cdot \nabla u_{h})^{\varepsilon}\cdot u_{h}^{\varepsilon} dxdt - \int_{0}^{T}\int_{\mathbb{R}^{3}} (u\cdot \nabla u_{3})^{\varepsilon}u_{3}^{\varepsilon} dxdt\\
	=&I+II.\ea\end{equation}
To deal with the first term,  the incompressiabe condition  helps us to get
\be\label{3.27}\ba &\int_{0}^{T}\int_{\mathbb{R}^{3}} (u\cdot \nabla u_{h}^{\varepsilon})\cdot u_{h}^{\varepsilon} dxdt\\
=& \int_{0}^{T}\int_{\mathbb{R}^{3}}\s(u u_{1}^{\varepsilon})\cdot u_{1}^{\varepsilon} dxdt+\int_{0}^{T}\int_{\mathbb{R}^{3}}\s(u u_{2}^{\varepsilon})\cdot u_{2}^{\varepsilon} dxdt=0,\ea\ee
and
\be\label{3.28}\ba
 &- \int_{0}^{T}\int_{\mathbb{R}^{3}} (u\cdot \nabla u_{h})^{\varepsilon}\cdot u_{h}^{\varepsilon} dxdt\\
= & -\int_{0}^{T}\int_{\mathbb{R}^{3}} \s(u u_{1})^{\varepsilon}\cdot u_{1}^{\varepsilon} dxdt-\int_{0}^{T}\int_{\mathbb{R}^{3}} \s(u u_{2})^{\varepsilon}\cdot u_{2}^{\varepsilon} dxdt.\\
\ea\ee
In combination with \eqref{3.27} and \eqref{3.28}, we can reformulate the first term as
\be\label{3.29}\ba
 I=&- \int_{0}^{T}\int_{\mathbb{R}^{3}} (u\cdot \nabla u_{h})^{\varepsilon}\cdot u_{h}^{\varepsilon} dxdt\\
 =&\int_{0}^{T}\int_{\mathbb{R}^{3}} \s[u u_{1} ^{\varepsilon}-(u u_{1})^{\varepsilon}]\cdot u_{1}^{\varepsilon} dxdt+\int_{0}^{T}\int_{\mathbb{R}^{3}} \s[u u_{2} ^{\varepsilon}-(u u_{2})^{\varepsilon}]\cdot u_{2}^{\varepsilon} dxdt.
\ea\ee
By the same token, we also have
\be\label{3.30}\ba
II&= - \int_{0}^{T}\int_{\mathbb{T}^{3}} (u\cdot \nabla u_{3})^{\varepsilon}u_{3}^{\varepsilon} dxdt=\int_0^T\int_{\mathbb{R}^3}\s [uu_3^\varepsilon-(u u_3)^\varepsilon]\cdot u_3^\varepsilon dxdt.
\ea\ee
Before going further, we assert that
\be\label{3.31}
\ba
\nabla_{h}u\in L^{p}(0,T;L^{ q}( \mathbb{R}^{3})),u\in L^{\f{2p}{p-1}}(0,T;L^{ \f{2q}{q-1}}( \mathbb{R}^{3}))\ea\ee
implies the energy equality \eqref{ei}.

Next, we prove that the terms on the right-hand side of \eqref{3.26} vanish as \(\varepsilon\to 0\). Contrary to the strategy adopted for the torus domain, we employ the anisotropic version of Lions' commutator estimates, instead of Constantin-E-Titi-type commutators, to handle these terms. Precisely, recall \eqref{3.31} and using the H\"older inequality, it gives
\be\label{3.32}\ba
 &\B|- \int_{0}^{T}\int_{\mathbb{R}^{3}} (u\cdot \nabla u_{h})^{\varepsilon}\cdot u_{h}^{\varepsilon} dxdt\B|\\
\leq&\|\s[u u_{1} ^{\varepsilon}-(u u_{1})^{\varepsilon}]\|_{L^{\f{2p}{p+1}}(0,T;L^{\f{2q}{q+1}}(\mathbb{R}^3))}\|u_{1}^{\varepsilon}\| _{L^{\f{2p}{p-1}}(0,T;L^{\f{2q}{q-1}}(\mathbb{R}^3))}\\
&~~~+\|\s[u u_{2} ^{\varepsilon}-(u u_{2})^{\varepsilon}]\|_{L^{\f{2p}{p+1}}(0,T;L^{\f{2q}{q+1}}(\mathbb{R}^3))}\|u_{2}^{\varepsilon}\| _{L^{\f{2p}{p-1}}(0,T;L^{\f{2q}{q-1}}(\mathbb{R}^3))}.
\ea\ee
Then,  by virtue of  lemma \ref{pLions} with the choices $u=u,f=u_1~\text{or}~u_2$, $p_1=p_4=\f{2p}{p-1}, q_1=q_4=\f{2q}{q-1}$, $p_3=p_2=p, q_3=q_2=q$ and $p=\f{2p}{p+1}, q=\f{2q}{q+1}$, we conclude
\be\label{3.33}\ba
&\s[u u_{1} ^{\varepsilon}-(u u_{1})^{\varepsilon}]\to 0, \quad\text{strongly~in } {L^{\f{2p}{p+1}}(0,T;L^{\f{2q}{q+1}}(\mathbb{R}^3))},~\text{as}~\varepsilon\to 0;\\
&\s[u u_{2} ^{\varepsilon}-(u u_{2})^{\varepsilon}] \to 0, \quad\text{ strongly~in } {L^{\f{2p}{p+1}}(0,T;L^{\f{2q}{q+1}}(\mathbb{R}^3))},~\text{as}~\varepsilon\to 0.
\ea\ee
Then together with \eqref{3.32} and \eqref{3.33}, we have
\be\label{3.34}
|I|=\B|- \int_{0}^{T}\int_{\mathbb{R}^{3}} (u\cdot \nabla u_{h})^{\varepsilon}\cdot u_{h}^{\varepsilon} dxdt\B|\to 0, ~\text{as}~\varepsilon \to0.
\ee
For term $II$, we note that the divergence-free condition together with \eqref{3.31} guarantees full regularity for $u_3$. Indeed, $\nabla_h u\in L^p(0,T;L^q(\mathbb R^3))$ implies $\nabla_h u_h\in L^p(0,T;L^q(\mathbb R^3))$. From the divergence-free constraint $\s u=0$, we have $\partial_3 u_3 = -\partial_h\cdot u_h$, which further yields $\nabla u_3\in L^p(0,T;L^q(\mathbb R^3))$.

Based on this, using the integration by parts and the H\"older inequality, we get
\be\label{3.35}\ba
|II|&= |\int_0^T\int_{\mathbb{R}^3}\s [uu_3^\varepsilon-(u u_3)^\varepsilon]\cdot u_3^\varepsilon dxdt|\\
&=|-\int_0^T\int_{\mathbb{R}^3} [uu_3^\varepsilon-u u_3+ u u_3-(u u_3)^\varepsilon]\cdot \nabla u_3^\varepsilon dxdt|\\
&\leq C\|uu_3^\varepsilon-u u_3+ u u_3-(u u_3)^\varepsilon\|_{L^{\f{p}{p-1}}(0,T;L^{\f{q}{q-1}}(\mathbb{R}^3))}\|\nabla u_3^\varepsilon \|_{L^{p}(0,T;L^{q}(\mathbb{R}^3))}\\
&\leq C\B( \|u\|_{L^{\f{2p}{p-1}}(0,T;L^{\f{2q}{q-1}}(\mathbb{R}^3))}\|u_3-u_3^\varepsilon\|_{L^{\f{2p}{p-1}}(0,T;L^{\f{2q}{q-1}}(\mathbb{R}^3))}\\
&~~~~~~~~~+\| u u_3 -(u u_3)^\varepsilon \|_{L^{\f{p}{p-1}}(0,T;L^{\f{q}{q-1}}(\mathbb{R}^3))}\B)\|\nabla u_3^\varepsilon \|_{L^{p}(0,T;L^{q}(\mathbb{R}^3))},
\ea\ee
which together with the standard properties of the mollification yields
\be\label{3.36}
|II|= |\int_0^T\int_{\mathbb{R}^3}\s [uu_3^\varepsilon-(u u_3)^\varepsilon]\cdot u_3^\varepsilon dxdt| \to 0, ~\text{as}~\varepsilon \to 0.\ee
Hence, combining \eqref{3.34} and \eqref{3.35}, we conclude the proof of assertion \eqref{3.31}.  Since the proof below is analogous to that in the previous subsection, Theorem \ref{the1.2} follows by repeating the arguments therein, and we omit the details.
\end{proof}

\section*{Acknowledgements}
  Ye was
 partially sponsored by the Training Program for Young Backbone Teachers in
 Higher Education Institutions of Henan Province (2024GGJS022). Wei was partially supported by the National Natural Science Foundation of China under grant (No. 12271433). Wang was sponsored by Natural Science Foundation of
Henan Province (No. 232300421077)  and supported by  the National Natural
Science Foundation of China under grant (No. 11971446 and No. 12071113).

{\bf Author Declarations statement}

The authors have no conflicts to disclose.

{\bf Data Availability}

This publication is supported by multiple datasets, which are openly available at locations cited in the reference section.


\begin{thebibliography}{00}














\bibitem{[BY2]}
H. Beirao da Veiga and  J. Yang,   On the energy equality for solutions to Newtonian and non-Newtonian fluids. Nonlinear Anal. 185 (2019), 388--402.




\bibitem{[BY]}H. Beirao da Veiga and J. Yang, On the Shinbrot's criteria for energy equality to Newtonian fluids: a simplified proof, and an extension of the range of application. Nonlinear Anal. 196 (2020), 111809, 4 pp.
\bibitem{[BY2]}H. Beirao da Veiga and J. Yang, The relations between H\"older continuity assumptions
on the direction of vorticity and energy equality. J. Geom. Anal. 35 (2025), 15.




\bibitem{[Berselli2]} L. C. Berselli, Energy conservation for weak solutions of incompressible fluid equations:
The H\"older case and connections with Onsager's conjecture. J. Differential Equations,
368 (2023), 350--375.



\bibitem{[BC]}
L. C. Berselli and  E. Chiodaroli,  On the energy equality for the 3D Navier-Stokes
equations. Nonlinear Anal. 192 (2020), 111704, 24 pp.

\bibitem{[BCS]}
L. C. Berselli and  E. Chiodaroli and R. Sannipoli, Energy conservation for 3D Euler and Navier-Stokes equations in a bounded domain: applications to Beltrami flows. J. Nonlinear Sci. 35 (2025), no. 1, Paper No. 10, 30 pp.

\bibitem{[BG]} L. C. Berselli and S.  Georgiadis,  Three results on the Energy conservation for the 3D
Euler equations.  Nonlinear Differ. Equ. Appl. 31, 33 (2024), 1--14.


\bibitem{[BKR]}
L. C. Berselli, A. Kaltenbach and M. R\r{u}\v{z}i\v{c}ka, Energy conservation for weak solutions of incompressible Newtonian fluid equations in H\"older spaces with Dirichlet boundary conditions in the half-space. Math. Ann. 391 (2025), no. 4, 5911--5940.


\bibitem{[BFWZ]}
Q. Bie, H. Fang, S. Wang and Y. Zhou, Stability and sharp decay for the 3D incompressible anisotropic Navier-Stokes equations with fractional horizontal dissipation. J. Differential Equations 463 (2026), Paper No. 114167, 38 pp.
\bibitem{[CG]}
C. Cao and Y. Guo, On the two-dimensional Navier-Stokes equations with horizontal viscosity.
J. Math. Anal. Appl. 557 (2026), no. 2, Paper No. 130352, 24 pp.
\bibitem{[CW1]}
C. Cao and J. Wu, Stability of the 3D Navier-Stokes equations with anisotropic dissipation. Nonlinearity. 38 (2025), no. 9, Paper No. 095012, 10 pp.
\bibitem{[CW]}
C. Cao and J. Wu, Global regularity for the two-dimensional anisotropic Boussinesq equations with vertical dissipation.
Arch. Ration. Mech. Anal.  208(2013), no. 3, 985--1004.
\bibitem{[CDGG]} J. Chemin, B. Desjardins, I. Gallagher and E. Grenier. Fluids with anisotropic viscosity. M2AN Math.
Model. Numer. Anal. 34(2000), no. 2, 315--335.



\bibitem{[CZZ]}
J. Chemin, P. Zhang and Z. Zhang, On the critical one component regularity for 3-D
Navier-Stokes system: general case. Arch. Ration. Mech. Anal. 224 (2017), 871--905.
\bibitem{[CCFS]}
A. Cheskidov,  P. Constantin, S. Friedlander and R. Shvydkoy, Energy conservation and Onsager's conjecture for the Euler equations. Nonlinearity, 21 (2008), 1233--1252.



\bibitem{[CL]}
   A. Cheskidov and X. Luo, Energy equality for the Navier-Stokes equations in weak-in-time Onsager spaces. Nonlinearity, 33 (2020), 1388--1403.

\bibitem{[CET]} P. Constantin,  W. E and   E. S. Titi,  Onsager's conjecture on the energy conservation for solutions of Euler's equation. Commun. Math. Phys. 165 (1994), 207--209.


\bibitem{[DW]}
J. Demmel and E. Wiedemann
Energy rigidity and weak-strong uniqueness for the 2D anisotropic Navier-Stokes equations.
 arXiv.2608.19931.

\bibitem{[DS0]}
C. De Lellis and  L.  J.  Sz\'ekelyhidi,
The Euler equations as a differential inclusion.
Ann. Math. 170 (2009), 1417--1436.

\bibitem{[DS1]}
C. De Lellis and  L.  J.  Sz\'ekelyhidi,   Dissipative continuous Euler flows. Invent Math. 193 (2013), 377--407.

\bibitem{[FGSW]}
E. Feireisl, P. Gwiazda, A. \'Swierczewska-Gwiazda and E. Wiedemann, Regularity and
energy conservation for the compressible euler equations. Arch Ration Mech Anal. 223
(2017), 1375--1395.
\bibitem{[FW2018]}
U. S. Fjordholm and E.  Wiedemann,  Statistical solutions and Onsager's conjecture.
Phys. D.      376-377 (2018), 259--265.
 \bibitem{[F]}
M. Fujii, Large time behavior of solutions to the 3D anisotropic Navie--Stokes equation. Nonlinear Anal. Real World Appl. 71 (2023), Paper No. 103821, 33 pp.







\bibitem{[Galdi]} G. P. Galdi, An introduction to the Navier-Stokes initial-boundary value problem, in: Fundamental Directions in Mathematical Fluid Mechanics, in: Adv. Math. Fluid Mech., Birkh\"auser, Basel, 2000,   1--70.

\bibitem{[GKN]}   V. Giri, H. Kwon  and M. Novack.
The $L^{3}$-based strong Onsager theorem. Ann. of Math. (2) 204 (2026), no. 1, 265--421.
\bibitem{[GR]}  V. Giri and R. O. Radu, The 2D Onsager conjecture: a Newton-Nash iteration. Invent. math. 238 (2024), 691--768.





\bibitem{[Isett]}
P. Isett,
A proof of Onsager's conjecture.
Ann. of Math.  188 (2018),  871--963.
\bibitem{[JTW]}
R. Ji, L. Tian and J. Wu, 3D anisotropic Navier-Stokes equations in $\mathbb{T}^2\times \mathbb{R}$ : stability and large-time behaviour. Nonlinearity. 36 (2023), no. 6, 3219--3237.

\bibitem{[JWY]}
R. Ji, J. Wu and W. Yang, Stability and optimal decay for the 3D Navier-Stokes equations with horizontal dissipation. J. Differential Equations 290 (2021), 57--77.












 \bibitem{[LS]}
T. M. Leslie and R. Shvydkoy. Conditions implying energy equality for weak solutions of the Navier--Stokes equations. SIAM J. Math. Anal.
 50 (2018), 870--890.
\bibitem{[Lions]}
J. L. Lions, Sur la r\'egularit\'e et l'unicit\'e des solutions turbulentes des \'equations de Navier Stokes, Rend. Semin. Mat. Univ. Padova, 30 (1960) 16--23.

\bibitem{[NNT]}
Q. Nguyen,  P. Nguyen and B. Tang,   Energy equalities for compressible Navier-Stokes equations. Nonlinearity 32 (2019),   4206--4231.
      \bibitem{[Onsager]}
L. Onsager,  Statistical hydrodynamics. Nuovo Cim. (Suppl.) 6 (1949), 279--287.

\bibitem{[Pedlosky]}
J. Pedlosky, Geophysical Fluid Dynamics. Springer-Verlag, Berlin, 1979.
	\bibitem{[Shinbrot]}
M. Shinbrot, The energy equation for the Navier-Stokes system. SIAM J. Math. Anal. 5 (1974) 948--954.
			
\bibitem{[Taniuchi]}
 Y. Taniuchi, On generalized energy equality of the Navier-Stokes equations.
Manuscripta Math. 94 (1997), 365--384.

\bibitem{[WZZ]}
W. Wang, L. Zhang and Z. Zhang,
On the interior regularity criteria of the 3-D navier-stokes equations involving two velocity components. Discrete Contin. Dyn. Syst.
  38 (2018),   2609--2627.

\bibitem{[WY]}
Y. Wang and Y. Ye, A general sufficient criterion for energy conservation in the Navier-Stokes system.  Math Meth Appl Sci.  46 (2023), 9268--9285.
\bibitem{[YJW]}
W. Yang, Q. Jiu and J. Wu, The 3D incompressible Navier-Stokes equations with partial hyperdissipation. Math. Nachr. 292 (2019), no. 8, 1823--1836.
\bibitem{[XZ]}
L. Xu and P. Zhang,  Enhanced dissipation for the third component of 3D anisotropic Navier-Stokes equations. J. Differential Equations 335 (2022), 464-496.


 \bibitem{[YWL]} Y. Ye,  Y. Wang and  J. Liu, Energy and helicity conservation in the incompressible ideal flows. Commun. Math. Sci. 23 (2025), no. 5, 1357--1377.
 \bibitem{[YWW]} Y. Ye,   W. Wei and Y. Wang, Energy equality in the isentropic compressible Navier-Stokes equations allowing vacuum. J. Differential Equations 338 (2022), 551--571.

\bibitem{[Yu]}X. Yu,  A note on the energy conservation of the ideal MHD equations. Nonlinearity. 22 (2009), 913--922.
\bibitem{[Zhang]}
		Z.	Zhang,  Remarks on the energy equality for the non-Newtonian fluids. J. Math. Anal. Appl. 480 (2019), no. 2, 123443, 9 pp.
		\bibitem{[ZZ]}
		P. Zhang and W. Zhu, Continuous dependence on initial data for the solutions of 3-D anisotropic Navier-Stokes equations. J. Funct. Anal. 288 (2025), no. 1, Paper No. 110689, 36 pp.
\bibitem{[Zheng]}
X. Zheng, A regularity criterion for the tridimensional Navier-Stokes equations in term
of one velocity component. J. Differential Equations 256 (2014), 283--309.

\bibitem{[ZW]}
D. Zhou and J. Wu, $H^1$ -uniqueness, stability and decay for anisotropic Navier-Stokes and Boussinesq equations.
J. Differential Equations,   453 (2026),  113930.

\end{thebibliography}
\end{document}